\documentclass[sn-mathphys,Numbered]{sn-jnl}

\usepackage{graphicx}%
\usepackage{float}%
\usepackage{multirow}%
\usepackage{amsmath,amssymb,amsfonts}%
\usepackage{amsthm}%
\usepackage{mathrsfs}%
\usepackage[title]{appendix}%
\usepackage{xcolor}%
\usepackage{textcomp}%
\usepackage{manyfoot}%
\usepackage{booktabs}%
\usepackage{algorithm}%
\usepackage{algorithmicx}%
\usepackage{algpseudocode}%
\usepackage{listings}%
\usepackage[latin1]{inputenc}

\theoremstyle{thmstyleone}%
\newtheorem{theorem}{Theorem}
\theoremstyle{thmstyletwo}%
\newtheorem{remark}{Remark}%
\newtheorem{lemma}{Lemma}%

\theoremstyle{thmstylethree}%
\newcommand{\argmax}{\raisebox{-1.8mm}{${{\displaystyle \rm arg\,\! max}
                     \atop {\scriptstyle 1 \le i \le n}}$}}

\begin{document}

\title[Article Title]{On the supremum and its location of the standardized uniform empirical process}


\author{\fnm{Dietmar} \sur{Ferger}}\email{dietmar.ferger@tu-dresden.de}



\affil{\orgdiv{Fakult\"{a}t Mathematik}, \orgname{Technische Universit\"{a}t Dresden}, \orgaddress{\street{Zellescher Weg 12-14}, \city{Dresden}, \postcode{01069}, \country{Germany}}}




\abstract{We show that the maximizing point and the supremum of the standardized uniform empirical process
converge in distribution. Here, the limit variable $(Z,Y)$ has independent components. Moreover, $Z$ attains the values zero and one with equal probability one half and $Y$ follows the Gumbel-distribution.
}




\maketitle

\section{Introduction and Main Result}
Let $X_1,\ldots,X_n$ be $n \in \mathbb{N}$ independent real random variables uniformly distributed on the interval $(0,1)$.
If $F_n$ denotes the corresponding empirical distribution function, then
$$
 V_n:=\sqrt{n} \sup_{0<t<1}\frac{|F_n(t)-t|}{\sqrt{t(1-t)}}
$$
is the supremum of the standardized uniform empirical process
$$
 Q_n(t)=\frac{|\sqrt{n}\{F_n(t)-t\}|}{\sqrt{t(1-t)}}, \; 0<t<1.
$$
Chibisov \cite{Chibisov} proves that $V_n$ converges to infinity in probability.
Whereas Jaeschke \cite{Jaeschke} shows that a certain affine-linear transformation of $V_n$ converges in distribution. More precisely,
let
\begin{equation} \label{anbn}
a_n=\sqrt{2 \log \log n} \quad \text{ and } \quad b_n=2 \log \log n+\frac{1}{2}\log \log \log n -\frac{1}{2} \log \pi.
\end{equation}
Then
\begin{equation} \label{limitJaeschke}
 a_n V_n - b_n \stackrel{\mathcal{D}}{\rightarrow} Y,
\end{equation}
where $Y$ follows the Gumbel-distribution, i.e. $\mathbb{P}(Y \le x)= \exp\{-2 e^{-x}\}, x \in \mathbb{R}.$ \\

Elementary arguments show that
\begin{equation} \label{sup}
 V_n = \sqrt{n} \max_{1 \le i \le n} \max \big\{\frac{\frac{i}{n}-X_{i:n}}{\sqrt{X_{i:n}(1-X_{i:n})}},\frac{X_{i:n}-\frac{i-1}{n}}{\sqrt{X_{i:n}(1-X_{i:n})}}\big\},
\end{equation}
where $X_{i:n}$ denotes the $i$-th order-statistic, $1 \le i \le n$.\\

In this paper we also consider the random variable $\tau_n \in (0,1)$, which maximizes $Q_n$ in the following sense:
\begin{equation} \label{deftaun}
 \tau_n = \min \{ t \in (0,1): \max\{Q_n(t),Q_n(t-)\}= V_n\}.
\end{equation}
The definition takes into account that $Q_n$ is right-continuous with left limits on its domain $(0,1)$. It turns out that the location of the supremum is given by $$\tau_n=X_{r:n},$$
where
\begin{equation} \label{indexr}
 r = \argmax \max \big\{\frac{\frac{i}{n}-X_{i:n}}{\sqrt{X_{i:n}(1-X_{i:n})}},\frac{X_{i:n}-\frac{i-1}{n}}{\sqrt{X_{i:n}(1-X_{i:n})}}\big\}.
\end{equation}

\vspace{0.4cm}
Moreover, $\tau_n$ is almost surely (a.s.) unique. Let $X_i^\prime :=1-X_i, 1 \le i \le n.$ Then $(X_1^\prime,\ldots,X_n^\prime) \stackrel{\mathcal{D}}{=} (X_1,\ldots,X_n)$ and $X_{i:n}^\prime = 1-X_{n-i+1:n}, 1 \le i \le n.$ Thus it follows from (\ref{sup}) and (\ref{indexr}) that
\begin{equation} \label{symmetric}
 (\tau_n,V_n) \stackrel{\mathcal{D}}{=} (1 - \tau_n,V_n).
\end{equation}

\vspace{0.4cm}
Our main result says that as $n$ tends to infinity the whole probability mass of $\tau_n$ is shifted equally into the boundary points zero and one of the
open unit interval. Moreover, $\tau_n$ and $V_n$ are asymptotically independent.\\

\begin{theorem} \label{Thm} For all $x \in \mathbb{R} \setminus \{0\}$ and $y \in \mathbb{R}$ the following limit-theorem holds:
\begin{equation} \label{limthm}
 \lim_{n \rightarrow \infty} \mathbb{P}(\tau_n \le x, a_n V_n -b_n \le y) = \mathbb{P}(Z \le x) \exp\{-2 e^{-y}\}= \mathbb{P}(Z \le x) \mathbb{P}(Y \le y),
\end{equation}
where $\mathbb{P}(Z=1)=\mathbb{P}(Z=0)=\frac{1}{2}.$
In particularly,  $$(\tau_n,a_n V_n-b_n) \stackrel{\mathcal{D}}{\rightarrow} (Z,Y)$$
with $Z$ and $Y$ independent.
\end{theorem}

\vspace{0.3cm}
Observe that there is a strong relationship between $\tau_n$ and $V_n$, which according to (\ref{deftaun}) is given by
$$
  \max\{Q_n(\tau_n),Q_n(\tau_n-)\} = V_n.
$$
Considering this fact, the asymptotic independence of $\tau_n$ and $V_n$ is quite surprising.\\

Indeed, this independence does not apply in the non-weighted case. More precisely consider
$$
W_n:= \sqrt{n} \sup_{0 \le t \le 1} |F_n(t)-t|= \sup_{0 \le t \le 1}|u_n(t)|,
$$
where $u_n(t) = \sqrt{n}(F_n(t)-t)$ denotes the uniform empirical process. Let $\sigma_n$ be the corresponding maximizing point, i.e.
$$
 \sigma_n:=\min \{t \in [0,1]: \max\{|u_n(t)|,|u_n(t-)|\}= W_n\}.
$$

Let $M$ be the sup-functional and $a$ the argmax-functional as defined in Lemma A.2 in Ferger \cite{Ferger1}. Then $(\sigma_n,W_n)= L(u_n)$, where $L:=(a,M)$ is a functional on the Skorokhod-space $D[0,1]$ with values in $\mathbb{R}^2$, which by Lemmas A.3 and A.4 in Ferger \cite{Ferger1} is measurable and continuous on $C[0,1]$. By Donsker's Theorem $u_n \stackrel{\mathcal{D}}{\rightarrow} B$ in $D[0,1]$, where $B$ is a Brownian bridge. Therefore, the Continuous Mapping Theorem (CMT) yields that
$$
 (\sigma_n,W_n) \stackrel{\mathcal{D}}{\rightarrow} (T,W),
$$
where $T=a(|B|)$ is the maximizing point of $|B|$, which can be shown to be a.s. unique. Moreover, $W=M(|B|)=\sup_{0 \le t \le 1}|B(t)|$ is the infinity-norm of the Brownian bridge.
According to Theorem 2.6 in  Ferger  \cite{Ferger0} the limit variable $(T,W)$ has density
\begin{equation} \label{density}
 f_{(T,W)}(x,y) = \sqrt{\frac{8}{\pi}} \psi(x,y) \psi(1-x,y),\quad x \in (0,1), y \ge 0,
\end{equation}
where
$$
 \psi(x,y) = y x^{-\frac{3}{2}} \sum_{j \ge 0} (-1)^j (2j+1) e^{-\frac{1}{2} (2j+1)^2 x^{-1}y^2}
$$

In view of (\ref{density}) the components $T$ and $W$ are far away from being independent.\\

\begin{remark}
Notice that the random index $r$ of the maximizing order-statistic $X_{r:n}$ depends on $n$. Since $\frac{r}{n} = F_n^{-1}(X_{r:n})$ and the uniform empirical quantile function $F_n^{-1}$ converges uniformly on $(0,1)$ to the identity with probability one, we obtain from Theorem \ref{Thm} and Slutsky's theorem that
$$
 \frac{r}{n} \stackrel{\mathcal{D}}{\rightarrow} Z.
$$
\end{remark}

\section{The proof}
Our proof of Theorem \ref{Thm} relies on a maximal-inequality for the process $Q_n$.\\

\begin{lemma} \label{maxinequalityQn} For every $a \in (0,\frac{1}{2}]$ and each $\lambda>0$ it follows that
\begin{equation} \label{maxinequality}
\mathbb{P}\big(\sup_{a \le t \le 1-a} \frac{|F_n(t)-t|}{\sqrt{t(1-t)}} > \lambda \big) \le 2 \; \lambda^{-2} \; n^{-1} \big(\log(\frac{1-a}{a})+1\big).
\end{equation}
\end{lemma}

\begin{proof} It is well-known that the process $\frac{F_n(t)-t}{1-t}, t \in [0,1),$ is a centered martingale with respect to the natural filtration, confer Gaenssler and Stute \cite{Gaenssler}, p.4. Therefore, $S(t):= (\frac{F_n(t)-t}{1-t})^2, t \in [0,1),$ is a non-negative sub-martingale.
From Lemma 3.3 in Ferger \cite{Ferger1} we can infer with $w(t)=\frac{1-t}{t}$ that

\begin{eqnarray}
\mathbb{P}\big(\sup_{a \le t \le 1/2} \frac{|F_n(t)-t|}{\sqrt{t(1-t)}} > \lambda \big)
&\le&\mathbb{P}\big(\sup_{a \le t \le 1/2} w(t) S(t) > \lambda^2 \big) \nonumber\\
&\le& \lambda^{-2}\{\int_a^{\frac{1}{2}} H(t) (-w)(dt)+w(\frac{1}{2}) H(\frac{1}{2})\}. \label{left}
\end{eqnarray}
Here, $H(t)=\mathbb{E}[S(t)]= (1-t)^{-2} \text{Var}[F_n(t)]= n^{-1} \frac{t}{1-t}.$ Since $-w^\prime(t)=t^{-2}$, the integral in (\ref{left}) simplifies to
$n^{-1} \log(\frac{1-a}{a})$ and $w(\frac{1}{2}) H(\frac{1}{2})=n^{-1}$. Thus we obtain from (\ref{left}) that
$$
 \mathbb{P}\big(\sup_{a \le t \le 1/2} \frac{|F_n(t)-t|}{\sqrt{t(1-t)}} > \lambda \big) \le \lambda^{-2} \; n^{-1} (\log(\frac{1-a}{a})+1).
$$
Now, the assertion follows, because $$\sup_{a \le t \le 1/2} \frac{|F_n(t)-t|}{\sqrt{t(1-t)}} \stackrel{\mathcal{D}}{=} \sup_{1/2 \le t \le 1-a} \frac{|F_n(t)-t|}{\sqrt{t(1-t)}}.$$
\end{proof}

For the proof of Theorem \ref{Thm} let $(\alpha_n)$ be a positive sequence converging to zero. Put $I:=(0,1)$ and $I_n:=[\alpha_n,1-\alpha_n]$. By assumption there exists an integer $n_0$ such that $\emptyset \neq I_n \subseteq I$ for all $n \ge n_0$. If
$$
 Y_n := \sup_{t \in I_n} Q_n(t) \quad \text{ and } \quad Z_n:=\sup_{t \in I \setminus I_n} Q_n(t),
$$
then $V_n= \sup_{t \in I} Q_n(t) = \max \{Y_n,Z_n\}.$\\

With the above maximal-inequality it follows that
$$
\mathbb{P}(a_n^{-1}|Y_n|>\epsilon)\le \epsilon^{-2} \frac{\log(\alpha_n^{-1})+1}{\log\log n} \quad \text{for all } \epsilon>0.
$$
Consequently, if for example $\alpha_n= (\log \log n)^{-1}$, then
\begin{equation} \label{zero}
 a_n^{-1} Y_n \stackrel{\mathbb{P}}{\rightarrow} 0.
\end{equation}
Put $\tilde{Y}_n:=a_n V_n-b_n$. Then $a_n^{-1} V_n= a_n^{-2}\tilde{Y}_n+ a_n^{-2} b_n$, and thus by (\ref{anbn}) and (\ref{limitJaeschke})
\begin{equation} \label{one}
 a_n^{-1} V_n \stackrel{\mathbb{P}}{\rightarrow} 1.
\end{equation}
Next, notice that
\begin{equation} \label{tauninIn}
\{\tau_n \in (\alpha_n,1-\alpha_n]\} \subseteq \{Z_n \le Y_n\}.
\end{equation}
To see this assume that $\tau_n \in (\alpha_n,1-\alpha_n]$, but $Z_n>Y_n$. Recall that by definition (\ref{deftaun}) the maximizer satisfies $\max\{Q_n(\tau_n),Q_n(\tau_n-)\}=V_n.$ We distinguish two cases. If $Q(\tau_n)=V_n$, then $V_n \le \sup_{\alpha_n< t \le 1-\alpha_n} Q_n(t) \le Y_n <Z_n \le V_n$, a contradiction. If $Q_n(\tau_n-)=V_n$, then $V_n = \lim_{t \uparrow \tau_n, t \in (\alpha_n,1-\alpha_n]} Q_n(t) \le \sup_{\alpha_n< t \le 1-\alpha_n} Q_n(t) \le Y_n <Z_n \le V_n$, once again a contradiction. This shows (\ref{tauninIn}) and therefore
\begin{eqnarray}
0&\le& \mathbb{P}(\tau_n \in (\alpha_n,1-\alpha_n]) \le \mathbb{P}(Z_n \le Y_n) \le \mathbb{P}(V_n=Y_n) \nonumber\\
 &\le& \mathbb{P}(V_n=Y_n, a_n^{-1} Y_n \le \epsilon)+\mathbb{P}(a_n^{-1} Y_n > \epsilon) \quad \forall \; \epsilon>0. \label{decomp}
\end{eqnarray}
The second probability in (\ref{decomp}) converges to zero for all positive $\epsilon$ by (\ref{zero}). The first probability in (\ref{decomp})
is less than or equal to
$$
 \mathbb{P}(V_n \le a_n \epsilon) \le \mathbb{P}(V_n \le a_n \epsilon, |a_n^{-1} V_n-1|\le \epsilon)+\mathbb{P}(|a_n^{-1} V_n-1| > \epsilon).
$$
Here, the second probability converges to zero by (\ref{one}) and the first probability actually vanishes whenever $\epsilon \in (0,1/2)$, because in this case the inequalities $V_n \le a_n \epsilon$ and $V_n \ge (1-\epsilon)a_n$ are mutually exclusive.
In summary, we have shown that
\begin{equation} \label{probtauninIn}
 \mathbb{P}(\tau_n \in (\alpha_n,1-\alpha_n]) \rightarrow 0, \; n \rightarrow \infty.
\end{equation}

Next, observe that
$$
 \mathbb{P}(\tilde{Y}_n \le y)=\mathbb{P}(\tau_n \le \alpha_n, \tilde{Y}_n \le y)+\mathbb{P}(\tau_n \in (\alpha_n,1-\alpha_n),\tilde{Y}_n \le y)+\mathbb{P}(\tau_n \ge 1-\alpha_n, \tilde{Y}_n \le y).
$$
Here, the second probability on the right side converges to zero as a consequence of (\ref{probtauninIn}). Moreover, the third probability is
equal to the first probability by symmetry (\ref{symmetric}). Thus
\begin{equation} \label{step1}
 \mathbb{P}(\tau_n \le \alpha_n, \tilde{Y}_n \le y)= \frac{1}{2} \mathbb{P}(\tilde{Y}_n \le y)+ \beta_n,
\end{equation}
where $\beta_n \rightarrow 0$. Finally, let $x \in (0,1)$. Then there exists some integer $n_1=n_1(x)$ such that $\alpha_n <x \le 1-\alpha_n$ for all $n \ge n_1$.
It follows that
\begin{equation} \label{step2}
 \mathbb{P}(\tau_n \le x,\tilde{Y}_n \le y)=\mathbb{P}(\tau_n \le \alpha_n,\tilde{Y}_n \le y)+\mathbb{P}(\tau_n \in (\alpha_n,x],\tilde{Y}_n \le y).
\end{equation}
Since the second probability on the right side in (\ref{step2}) is less that or equal to $\mathbb{P}(\tau_n \in (\alpha_n,1-\alpha_n])$ for all $n \ge n_1$, it converges to zero by (\ref{probtauninIn}). Therefore, it follows from (\ref{step2}), (\ref{step1}) and (\ref{limitJaeschke}) that
$$
 \mathbb{P}(\tau_n \le x,\tilde{Y}_n \le y) \rightarrow \frac{1}{2} \exp\{-2 e^{-y}\} \quad \forall \; x \in (0,1) \; \forall \; y \in \mathbb{R}.
$$
For $x<0$ or $x \ge 1$ the validity of (\ref{limthm}) is trivial, because $\tau_n \in (0,1)$ a.s. and (\ref{limitJaeschke}) holds. As to the second part of Theorem \ref{Thm} notice that in particularly we have
convergence at every continuity point of $Z$ so that we can apply a well-known characterisation for distributional convergence of random variables in the euclidian space $\mathbb{R}^d$, confer for instance Proposition 5.58 in Witting and M\"{u}ller-Funk \cite{Witting}.





\end{document}